\documentclass{amsart}

\usepackage{amsmath, amsthm, amssymb}
\usepackage{epsfig}
\usepackage{graphicx}
\usepackage{array}
\usepackage{amsmath}
\usepackage{amsfonts}
\usepackage[all]{xy}
\usepackage{amsfonts}
\usepackage{amsmath}%
\usepackage{amsfonts}%
\usepackage{amssymb}%
\usepackage{graphicx}%------------------------------------------------------------

\usepackage{amsfonts}
\usepackage{amssymb}
\usepackage{amsmath,amsxtra,amssymb}
\numberwithin{equation}{section}

\makeatletter
\mathchardef\hugecheck="7014
\newcommand\hugesize{\@setfontsize\hugesize{25pt}{0}}
\newcommand\smallhugesize{\@setfontsize\smallhugesize{20pt}{0}}

\def\specialchecksmall{\mbox{\hbox
to0pt{\raisebox{-4pt}{\smallhugesize$\hugecheck$}}
                                 $\kern-2.7pt\otimes$}}
\makeatother

\newtheorem{theorem}{Theorem}[section]

\newtheorem{corollary}[theorem]{Corollary}
\newtheorem{proposition}[theorem]{Proposition}

\newcommand{\field}[1]{\mathbb{#1}}
\newcommand{\C}{{\field{C}}}

\newcommand{\ra}{\rightarrow}
\newcommand{\id}{{\iota}}
\newcommand{\ot}{{\,\otimes\,}}
\newcommand{\om}{{\omega}}
\newcommand{\tp}{{\,\widehat{\otimes}\,}}
\newcommand{\vtp}{\,{\overline{\otimes}\,}}

\newcommand{\ftp}{{\,\overline {\otimes}_{\mathcal{F}}\,}}

\newcommand{\h}{{\mathcal H}}

\newcommand{\B}{{\mathcal{B}}}
\newcommand{\LL}{{L_{\infty}(\G)}}
\newcommand{\LO}{{L_{1}(\G)}}
\newcommand{\LT}{{L_{2}(\G)}}

\newcommand{\CU}{{C_u(\G)}}

\newcommand{\loneqg}{L_1(\mathbb{G})}
\newcommand{\linfqg}{L_\infty(\mathbb{G})}

\newcommand{\PU}{\mathcal {P}}
\newcommand{\G}{\mathbb G}

\def\proclaim #1. #2\par{\medbreak
\noindent{\bf#1.\enspace}{\sl#2}\par\medbreak}
\allowdisplaybreaks

\begin{document}

\title [Realization of quantum group Poisson boundaries as crossed products]
{Realization of quantum group Poisson boundaries \\ as crossed products}
\author{Mehrdad Kalantar, Matthias Neufang and Zhong-Jin Ruan}
\address{School of Mathematics and Statistics,
         Carleton University, Ottawa, Ontario, Canada K1S 5B6}
\email{mkalanta@math.carleton.ca}
\address{School of Mathematics and Statistics,
         Carleton University, Ottawa, Ontario, Canada K1S 5B6}
\email{mneufang@math.carleton.ca}
\address{Universit\'{e} Lille 1 - Sciences et Technologies,
U.F.R. de Math\'{e}matiques, Laboratoire de Math\'{e}matiques Paul Painlev\'{e} - UMR CNRS 8524,
59655 Villeneuve d'Ascq C\'{e}dex, France}
\email{matthias.neufang@math.univ-lille1.fr}

\address{Department of Mathematics,
         University of Illinois, Urbana, IL 61801, USA}
\email{ruan@math.uiuc.edu}
\thanks{The second and third author were partially supported by NSERC 
and the Simons Foundation, respectively.} 

%============================
%  GENERAL INFORMATION
%============================
%In fact 2010 classification
%\date{\today}

%============================
%  ABSTRACT
%============================
\begin{abstract}
For a locally compact quantum group $\G$, consider the convolution action of a
quantum probability measure $\mu$ on $\LL$. As shown by Junge--Neufang--Ruan, this action 
has a natural extension to a Markov map on $\B(\LT)$. 
We prove that the Poisson boundary of the latter 
can be realized concretely as the von Neumann crossed product of the Poisson boundary associated with 
$\mu$ under the action of $\G$ induced by the coproduct. 
This yields an affirmative answer, for general locally compact quantum groups,
to a problem raised by Izumi (2004) in the commutative situation,
in which he settled the discrete case, and unifies earlier results of Jaworski, Neufang and Runde.
\end{abstract}

\maketitle

\section{Introduction and preliminaries}

Poisson boundaries and harmonic functions
have played a very important role in the study of random
walks on discrete groups, and more generally in harmonic analysis and
ergodic theory on locally compact groups
(see for instance Furstenberg's seminal work \cite{Furst}). 
The noncommutative version of this concept can be considered
in two different directions:
one replaces $L_\infty(G)$ by certain quantum groups, such as
its dual (quantum) group von Neumann algebra $VN(G)$ or
quantum groups arising from mathematical physics;
in another direction, one can perform the quantum mechanical passage from $L_\infty(G)$ to $\B(L_2(G))$. 

In the first direction, Chu--Lau \cite{Chu-Lau} studied the case of $VN(G)$, and 
Poisson boundaries over (discrete) quantum groups 
were first investigated by Izumi \cite{I}, in particular for the
dual of Woronowicz's compact quantum group $SU_q(2)$.
Izumi's results were further generalized to other discrete quantum groups by Izumi--Neshveyev--Tuset, Tomatsu, 
Vaes, Vander Vennet and Vergnioux (\cite{INT}, \cite{Tom}, 
\cite{VV2}, \cite{VV3}). 

The second way of quantization has been studied by Neufang--Ruan--Spronk in 
\cite{NRS}, based on a natural isometric representation $\Theta$ of the measure algebra $M(G)$
of a locally compact group $G$, 
as completely bounded maps on $\B(L_2(G))$, extending the convolution action of measures on $L_\infty(G)$ 
(cf. \cite{Gha}, \cite{Sto}). The structure of the Poisson boundary at the level of $\B(L_2(G))$, 
associated with a probability measure, has been analyzed by Izumi in \cite{II, III} for discrete groups, 
and by Jaworski--Neufang for locally compact groups in \cite{JN}, answering a question raised in \cite{II}. 
Dually, \cite{NRS} investigates the extension of the pointwise action of 
the completely bounded Fourier algebra multipliers $M_{cb} A(G)$, from $VN(G)$ to the level of $\B(L_2(G))$; the 
structure of the corresponding Poisson boundary associated with positive definite functions has been described in 
\cite{NR}, providing a noncommutative analogue of the situation considered in \cite{Chu-Lau}. 

In this paper, we combine these two quantization methods: for a locally compact quantum group $\G$,
we consider the natural extension of the action of a quantum probability
measure $\mu$ on $L_\infty(\G)$ to $\B(L_2(\G))$, as introduced and
studied by Junge--Neufang--Ruan in \cite{JNR}. We completely describe the structure of the Poisson
boundary $\h_{\Theta(\mu)}$ of the induced Markov operator $\Theta(\mu)$.
The main result (Theorem \ref{main}) of the paper
gives a concrete realization of $\h_{\Theta(\mu)}$
as the (von Neumann) crossed product of ${\mathcal H}_\mu$ with 
the quantum group $\G$ under a natural action.

To obtain an action of the quantum group $\G$ on the Poisson boundary,
it is natural to consider the restriction of the co-multiplication to
the latter.
In our setting, a difficulty lies in the fact that in general $\LL$ need not have the slice map property ($S_\sigma$);
for instance, given a discrete group $G$, $VN(G)$ has $S_\sigma$ if and only if $G$ has the approximation property (AP), cf. \cite{HK}.
Note that this problem does not arise in the case of commutative or discrete quantum groups $\G$ since $\LL$ then has $S_\sigma$.
We solve this problem by using the Fubini product, a tool from operator space theory.

In \cite[Theorem 4.1]{II} Izumi proved a crossed product formula for
Poisson boundaries associated with $\ell_\infty(G)$, where $G$ is a countable discrete
group, and as a direct consequence he concluded that the natural action of $G$ on
its Poisson boundary is amenable. Hence, he provided a new proof for this well-known result
of Zimmer \cite{Zim}, and therefore raised the question if such an identification
result extends to all second countable locally compact groups \cite[Problem 4.3]{II}.
This problem was answered affirmatively in \cite{JN}.
The present paper answers Izumi's
question even in the setting of locally compact quantum groups $\G$. 
As a corollary, we conclude that the crossed product von Neumann algebra of
the natural action of $\G$ on its (noncommutative) Poisson boundary is injective.

\par
We recall some definitions and preliminary results that will be used in the sequel.
For more details on locally compact quantum groups we refer the reader to \cite{KV1}, \cite{KV2}.

A {\it locally compact quantum group} $\G$ is a
quadruple $(\LL, \Gamma, \varphi, \psi)$, where $\LL$ is a
von Neumann algebra,
$\Gamma: \LL\to \LL \bar\otimes \LL$
is a co-associative co-multiplication,
and $\varphi$ and  $\psi$ are (normal faithful semi-finite) left, respectively, right
invariant weights on $\LL$,
called the \emph{Haar weights}. The weight $\psi$ determines a Hilbert space $L_{2}(\G)= L_{2}(\G, \psi)$, and
we obtain the \emph{right fundamental unitary operator} $V$
on $L_{2}(\G) \otimes L_{2}(\G)$, which satisfies the
\emph{pentagonal relation} 
$V_{12} V_{13}V_{23} = V_{23} V_{12}$.
Here we used the leg notation
$V_{12}=V\ot 1$, $V_{23} = 1\ot V$, and $V_{13} = (\id\ot\chi)V_{12}$,
where $\chi(x\ot y) = y\ot x$ is the flip map on $\B(H\otimes K)$.
We will also use the flip map $\sigma$ defined on $H\otimes K$.
The fundamental unitary operator induces a coassociative comultiplication
\begin{equation}
\label {F.general}
\tilde \Gamma: x \in \B(L_{2}(\G)) \to \tilde \Gamma(x) = V(x\otimes 1)V^{*}
\in \B(L_{2}(\G)\otimes L_{2}(\G))
\end{equation}
on $\B(L_{2}(\G))$, for which   we have 
$\tilde \Gamma _{|L_{\infty}(\G)}= \Gamma$.

Let $\loneqg$ be the predual of $\linfqg$. Then the pre-adjoint of
$\Gamma$ induces  an associative completely
contractive multiplication
\begin{equation}
\label {F.mul} \star  :  f \otimes g \in 
L_{1}(\G)\tp L_{1}(\G) \to f \star g = (f \otimes
g) \Gamma \in L_{1}(\G)
\end{equation}
on $\loneqg$.
The convolution actions $x\star\mu:=(\id\otimes\mu)\Gamma(x)$ and $\mu\star x:=(\mu\otimes\id)\Gamma(x)$
are normal completely bounded maps on $\LL$.

We denote by $C_{0}(\G)$ the \emph{reduced quantum group $C^*$-algebra} , which
is a weak$^*$ dense $C^*$-subalgebra of $\LL$.
Let $M(\G)$ denote the operator dual $C_{0}(\G)^{*}$.
There exists  a completely contractive multiplication on $M(\G)$ given by
the convolution
\[
\star : \mu\ot \nu \in M(\G)\tp M(\G)\mapsto \mu \star \nu 
= \mu (\id\otimes \nu)\Gamma = \nu (\mu \otimes \id)\Gamma
\in M(\G)
\]
such that  $M(\G)$ contains $\loneqg$ as a norm closed two-sided ideal.

We denote by $\CU$ the \emph{universal quantum group $C^*$-algebra} of $\G$ (see \cite {Kus} for details).
There is a comultiplication
\[
\Gamma_u :\CU\ra M(\CU\otimes_{min}\CU), 
\]
and the operator dual $M_u(\G)=C_{u}(\G)^{*}$, which can be regarded
as the space of all {\it quantum measures} on $\G$, is a
unital completely contractive Banach algebra with multiplication given by
\[
\om \star_{u} \mu = \om (\id\otimes \mu)\Gamma_{u} = \mu(\om\otimes \id)\Gamma_{u}\,.
\]
Moreover, the convolution algebras $M(\G)$ and $\LO$
can be canonically identified with norm closed two-sided ideals in $M_u(\G)$.
Therefore, for each $\mu\in M_u(\G)$, we obtain
a pair of completely bounded maps
\begin{equation}
\label {F.multi}
\mathfrak{m}^{l}_{\mu}(f) = \mu \star f ~~~~~~ \ \ \mbox{and } ~~~~~~ \ \
\mathfrak{m}^{r}_{\mu}(f) = f \star \mu
\end{equation}
on $\LO$ with $\mbox{max}\{\|\mathfrak{m}^{l}_{\mu}\|_{cb},
\|\mathfrak{m}^{r}_{\mu}\|_{cb}\}\le \|\mu\|$.
The adjoint map $\Phi_{\mu}= (\mathfrak{m}^{r}_{\mu})^{*}$ 
is a normal completely bounded map on $\LL$ 
satisfying the covariance condition
\begin{equation}
\label {F.Phi}
\Gamma \circ \Phi_{\mu} = (\id \otimes \Phi_{\mu})\circ \Gamma, 
\end{equation}
or equivalently,
$\Phi_{\mu}(f \star x) = f\star\Phi_{\mu}(x)$
for all $x \in \LL$ and $f\in \LO$.
We are particularly interested in the case when $\mu$ is a state in $M_{u}(\G)$.
In this case, $\Phi_{\mu}$ is unital completely positive,
i.e., a Markov operator, on $\LL$.

\section{Concrete realization of the Poisson boundary in $\B(\LT)$}

In the following, $\G$ denotes a general locally compact quantum group.
We write $\PU_u(\G)$ for the set of all states on $\CU$ (i.e., the `quantum probability measures').
We consider the space of fixed points $\h^{\mu} = \{x\in L_{\infty}(\G): \Phi_{\mu}(x)=x\}$.
It is easy to see that $\h^{\mu}$ is a weak* closed operator system in $L_{\infty}(\G)$.
In fact, we obtain a natural von Neumann algebra product on this space.
Let us recall this construction for the convenience of the reader (cf. \cite [Section 2.5]{I}).

We first define a projection
${\mathcal E}_{{\mu}} : L_{\infty}(\G) \to L_{\infty}(\G)$ of norm one by 
\begin{equation}\label{11}
{\mathcal E}_{\mu} (x) = w^*-\lim_{\mathcal{U}} \frac{1}{n}\sum_{k=1}^{n} \Phi_{\mu}^k (x)
\end{equation}
with respect to a free ultrafilter $\mathcal{U}$ on $\mathbb{N}$.
It is easy to see that $\h^{\mu} = {\mathcal E}_{\mu}(L_{\infty}(\G))$, and that
the Choi-Effros product
\begin{equation}
\label {2l}
x\circ y  = {\mathcal E}_{\mu} (xy)
\end{equation}
defines a von Neumann algebra product on $\h^{{\mu}}$.
We note that this product is independent of the choice of the free ultrafilter
$\mathcal{U}$ since every completely positive isometric linear isomorphism between
two von Neumann  algebras is a $^*$-isomorphism.
To avoid confusion, we denote by $\h_{\mu} = (\h^{\mu}, \circ)$
this von Neumann algebra, and we call
$\h_{\mu}$ the \emph{Poisson boundary} of $\mu$.

It follows from \cite[Theorem 4.5]{JNR} that the Markov operator 
$\Phi_{\mu}$ has a  unique weak* continuous (unital completely positive) extension
$\Theta(\mu)$ to $\B(L_{2}(\G))$ such that 
\begin{equation}
\label {F.rep}
\tilde \Gamma \circ \Theta(\mu) = (\id\otimes  \Phi_{\mu})\circ \tilde \Gamma.
\end{equation}
Similarly to (\ref{11}) we obtain a projection ${\mathcal E}_{\Theta(\mu)}$ of norm one
on $\B(L_{2}(\G))$ and a von Neumann algebra product on  $\h^{\Theta(\mu)}$.
We denote this von Neumann algebra by $\h_{\Theta(\mu)}$.
Our main result (Theorem \ref{main}) shows that there is a left $\G$ action
on $\h_{\mu}$ such that $\h_{\Theta(\mu)}$ is *-isomorphic to the von Neumann
algebra crossed product of $\h_{\mu}$ by $\G$.

Unlike in the classical and the discrete settings, in the case of general
locally compact quantum groups, the fact that the coproduct induces
an action of $\G$ on a Poisson boundary is not trivial. So, first we need to prove this result.
For this purpose, we need to recall the Fubini product for weak* closed operator spaces
on Hilbert spaces. Suppose that  $V$ and $W$ are weak$^*$ closed subspaces of $\B(H)$ and $\B(K)$,
respectively. We define the \emph{Fubini product} of $V$ and $W$ to be the space
\[
V\ftp W= \{X\in\B(H)\vtp\B(K) : (\omega\ot\id)(X)\in W \ \text{and} \ (\id\ot\phi)(X)\in V \ 
\text{for all} \ \omega\in \B(H)_*, \phi\in \B(K)_*\} .
\]
In this case, $V_{*} = \B(H)_{*}/V_{\perp}$ and $W_{*} = \B(K)_{*}/W_{\perp}$
are operator preduals of $V$ and $W$, respectively.
It is known from \cite[Proposition 3.3]{Ruan1992} 
(see also \cite {EKR} and  \cite [$\S 7.2$]{ERbook})
that the Fubini product $V\ftp W$ is a weak$^*$
closed subspace of $\B(H\otimes K)$ such that we have the
weak$^*$ homeomorphic completely isometric isomorphism
\begin{equation}
\label {F.fubini}
V\ftp W = (V_{*}\hat \otimes W_{*})^{*},
\end{equation}
where $V_{*}\hat \otimes W_{*}$ is the operator space projective tensor product of 
$V_{*}$ and $W_{*}$.
In particular, if  $M$ and $N$ are von Neumann algebras,  
the Fubini product coincides with the 
von Neumann algebra tensor product, i.e., we have 
\[
M\ftp N = M\vtp N.
\]

It is also known from operator space theory that there is a canonical   completely 
isometrically identification
\begin{equation}
\label {F.fubini2}
(V_{*}\hat \otimes W_{*})^{*}= {\mathcal {CB}}(V_{*}, W)
\end{equation}
given by the left slice maps. 
Now, if $W_{1}$ and $W_{2}$ are dual operator spaces and 
$\Psi: W_{1}\to W_{2}$ is a (not necessarily weak* continuous) 
completely bounded map, we can apply (\ref {F.fubini}) and (\ref {F.fubini2})
to obtain a completely bounded map
\[
\id\otimes \Psi : V \ftp W_{1}\to V\ftp W_{2}
\]
such that 
\[
(\omega \otimes \id)(\id\otimes \Psi) (X) = \Psi ((\om \otimes \id) (X))
\]
for all  $X\in V\ftp W_1$ and $\omega \in V_{*}$.
It is easy to see that we have $\|\id\otimes \Psi\| _{cb}= \|\Psi\|_{cb}$,
and if $\Psi$ is a completely isometric isomorphism 
(respectively, completely contractive projection)  then so is $\id \ot \Psi$.

Now since $\h^{\mu}$ and $\h_{\mu}= (\h^{\mu}, \circ)$ have the same predual,
 the identity map $\id_{\mu}$ is a 
weak$^*$  homeomorphic and completely isometric  isomorphism  from 
the weak$^*$  closed operator system $\h^{\mu}$ onto the von Neumann algebra
$\h_{\mu}$.  So we obtain the weak*  homeomorphic  and completely isometric isomorphism  
\[
\id \otimes \id_{\mu}: \LL\ftp \h ^{\mu}   \to  \LL\ftp \h_{\mu}.
\]
Since both   $L_{\infty}(\G)$ and $\h_{\mu}$ are von Neumann algebras,
we can identify $ \LL\ftp \h_{\mu}$ with  the von Neumann algebra
$ \LL\vtp \h_\mu$.
We note  that since ${\mathcal E}_{\mu}$ is a  (not necessarily normal) 
 projection of norm one  from  $L_{\infty}(\G) $ onto 
 $\h^{\mu}\subseteq L_{\infty}(\G) $, the map
$\id \otimes {\mathcal E}_{\mu}$ defines a projection of norm one from
 $L_{\infty}(\G)  \vtp L_{\infty}(\G) = L_{\infty}(\G)  \ftp L_{\infty}(\G)$ onto 
$\LL\ftp \h^{\mu} \subseteq L_{\infty}(\G)  \vtp L_{\infty}(\G)$ and thus 
induces a Choi-Effros product such that  $\LL\ftp \h^{\mu}$ becomes a von Neumann algebra.
It turns out that (up to the above identification) this von Neumann algebra is 
exactly equal to $\LL \vtp \h_{\mu}$.

\begin{proposition} \label{lemma1}
For any $\mu\in \PU_u(\G)$ the restriction of  $\Gamma$
  to $\h_{\mu}$ induces a left action $\Gamma_\mu$ 
of $\G$ on the von Neumann algebra ${\mathcal H}_\mu$.
\end{proposition}
\begin{proof} 
Let us first show that  $\Gamma (\h^{\mu})\subseteq \LL\ftp \h^{\mu}$.
Given $x\in \h^{\mu}$ and $f \in \LO$,  we have 
\begin{eqnarray*}
\Phi_{\mu}\big((f\otimes \id) \Gamma(x)\big) = 
(f\otimes \id) (\id \otimes \Phi_{\mu})\Gamma(x) 
= (f\otimes \id)\Gamma(\Phi_{\mu}(x)) = 
(f\otimes \id) \Gamma(x).
\end{eqnarray*}
This shows that $(f\otimes \id) \Gamma(x)$ is  contained in $\h^{\mu}$ for all $f \in \LO$.
On the other hand,  we clearly  have $(\id\ot g)\Gamma(x)\in L_{\infty}(\G)$ for all $g\in(\h_{\mu})_*$.
Hence, $\Gamma(x)\in \LL\ftp \h^{\mu}$.

It is clear that $\Gamma\mid_{\h^\mu} : \h^{\mu}\to  \LL\ftp \h^{\mu}$
is a normal   injective  unital completely positive isometry.
This induces the map $\Gamma_\mu : \h_\mu\to\LL\vtp\h_\mu$ given by
$\Gamma_\mu = (\id\ot\id_\mu)\circ\Gamma\circ{\id_\mu}^{-1}$.
It suffices to show that $\Gamma_{\mu}$ is an algebra homomorphism with respect to the
corresponding Choi-Effros products on $\h_{\mu}$ and $\LL\ftp \h_{\mu}$.
Given $x, y \in \h_{\mu}$, we now use  (\ref{11}) and (\ref {2l}) to obtain that
\begin{eqnarray*}
\Gamma_{\mu}(x \circ y) &=& \Gamma ({\mathcal E} _{{\mu}}(xy)) = 
\Gamma(w^*-\lim_{\mathcal U}\frac{1}{n}\sum_{k=1}^{n} \Phi_{\mu}^k(xy))
= w^*-\lim_{\mathcal U}\frac{1}{n}\sum_{k=1}^{n}\Gamma(\Phi_{\mu}^{k}(xy))\\
&=& w^*-\lim_{\mathcal U}\frac{1}{n}\sum_{k=1}^{n} 
( \id \otimes \Phi_{\mu}^{k})(\Gamma(xy))
= ( \id \otimes {\mathcal E} _{{\mu} })(\Gamma(x) \Gamma(y))\\
&=&
 \Gamma_{\mu}(x)\circ\Gamma_{\mu}(y) \in \LL\vtp \h_{\mu}.
\end{eqnarray*}
Since $\Gamma$ is a comultiplication on $L_{\infty}(\G)$, it is clear 
that $\Gamma_{\mu}$ satisfies 
\[
(\id \otimes \Gamma_{\mu})\circ \Gamma_{\mu} = (\Gamma \otimes \id)\circ \Gamma_{\mu}.
\]
So $\Gamma_\mu$ defines a left action.
\end{proof}

As we discussed above,  we have a projection 
${\mathcal E}_{\Theta(\mu)}$ of norm  one  from  
$\B(L_{2}(\G))$ onto $\h^{\Theta(\mu)}$ and  obtain  the  Choi-Effros
von Neumann algebra product on $\h_{\Theta(\mu)}$ given by
\[
X \circ Y = {\mathcal E}_{\Theta(\mu)}(XY)
\]
for $X, Y \in \h_{\Theta(\mu)}= (\h^{\Theta(\mu)}, \circ)$.
It is easy to see that the restriction of ${\mathcal E}_{\Theta(\mu)}$ to $L_{\infty}(\G)$ 
is equal to ${\mathcal E}_{\mu}$ 
and  $\h_{\mu}$ is a von Neumann subalgebra of $\h_{\Theta(\mu)}$.
Recall that  $\tilde \Gamma$ denotes the comultiplication on $\B(L_{2}(\G))$ 
defined in (\ref {F.general}).

\begin{proposition} 
For any $\mu\in \PU_u(\G)$
the restriction of $\tilde\Gamma$ to ${\mathcal H}_{\Theta(\mu)}$ induces a normal 
injective unital $*$-homomorphism 
\[
\tilde \Gamma_{{\Theta(\mu)}} :{\mathcal H}_{\Theta(\mu)}\rightarrow 
\B(L_2(\G)) \vtp {\mathcal H}_\mu
\]
between von Neumann algebras.
Moreover, the restriction of $\tilde \Gamma_{\Theta(\mu)}$ to $\h_{\mu}$ is equal to $\Gamma_{\mu}$.
\end{proposition}
\begin{proof}
For any $x \in \h_{\Theta(\mu)}$, we have $\tilde \Gamma(x)\in 
 \B(L_2(\G)) \vtp \LL$.
We can apply (\ref {F.rep}) to get 
\[
\Phi_{\mu}\left((\om \otimes \id)\tilde \Gamma(x)\right)
= (\om \otimes \id)(\id \otimes \Phi_{\mu}) \tilde \Gamma(x)
=(\om \otimes \id)(\tilde \Gamma(\Theta(\mu)(x))) 
= (\om \otimes \id)(\tilde \Gamma(x))
\]
for all $\omega \in \B(L_{2}(\G))_{*}$.
This shows that  $\tilde \Gamma(\h_{\Theta(\mu)})\subseteq  \B(L_2(\G)) \ftp \h_{\mu}
=\B(L_2(\G))\vtp \h_{\mu}$.
The rest of proof is similar to that given in the proof of Proposition \ref {lemma1}.
\end{proof}

For $\mu\in \PU_u(\G)$, let $\Gamma_{\mu}$ be the left action of $\G$ on the 
von Neumann algebra  $\h_{\mu}$ given in Proposition \ref {lemma1}.
The crossed product $\G  \ltimes_{\Gamma_{\mu}} \h_{\mu}$ 
 is  defined to be the von Neumann algebra 
 \[
 \G  \ltimes _{\Gamma_{\mu}} \h_{\mu} = [ \Gamma_{\mu}(\h_{\mu}) \cup 
 (L_{\infty}(\hat \G)\otimes 1) ]''
 \]
in $\B(L_2(\G))\vtp \h_{\mu}$.  The following result, which is crucial  for us, 
even holds in the setting of measured quantum groupoids 
 (cf.\ \cite[Theorem 11.6]{Enock}). 

\begin{theorem}
Let $\mu\in \PU_u(\G)$.
Denote by $\chi$ the flip map $\chi (a \otimes b) = b \otimes a$.
Then  
\[
\beta :
x \in  \B(L_2(\G)) \vtp \h_{\mu}\ra 
 (\sigma V^{*} \sigma\otimes 1)\big((\chi\otimes \id)(\id \otimes \Gamma_{\mu})(x)\big)
(\sigma V \sigma\otimes 1) \in  \LL\vtp \B(L_2(\G)) \vtp \h_{\mu}
\]
 defines a left action of $\G$ on the von Neumann algebra 
  $ \B(L_2(\G))\vtp \h_{\mu}$, and we have
\[
\G  \ltimes_{\Gamma_{\mu}}  \h_{\mu} = (\B(L_2(\G))\vtp \h_{\mu})^\beta, 
\]
where $ (\B(L_2(\G))\vtp \h_{\mu})^\beta 
= \{ y\in \B(L_2(\G))\vtp \h_{\mu}: 
\beta (y) = 1 \otimes y \}$ is the fixed point algebra of $\beta$.
\end{theorem}

Now we can prove the main theorem of this paper. 
Applied to the case $\G = L_\infty(G)$, even without the assumption of second countability, 
our theorem yields \cite[Proposition 6.3]{JN} which provided the answer to Izumi's original question. 
Specializing to the case $\G = VN(G)$, our result implies 
the main theorem of \cite{NR} where it was assumed either
that $G$ has the approximation property, or $\mu$ belongs to the Fourier algebra $A(G)=L_1(\G)$. 

\begin{theorem} \label{main}
Let $\G$ be a locally compact quantum group and let $\mu\in \PU_u(\G)$.
The induced map
\[
\tilde \Gamma_{{\Theta(\mu)}} : {\mathcal H}_{\Theta(\mu)}\rightarrow
\B(L_2(\G)) \vtp {\mathcal H}_\mu
\]
defines a von Neumann algebra isomorphism between
 ${\mathcal{H}}_{\Theta(\mu)}$ and $\G  \ltimes _{\Gamma_{\mu}} \h_{\mu}$.
\end{theorem}
\begin{proof} 
Since $V \in L_{\infty}(\hat \G)'\vtp L_{\infty}(\G)$,
it is easy to see  from (\ref {F.general}) and (\ref {F.rep}) that we have
\begin{eqnarray*}
\Theta(\mu)(\hat x)\otimes 1 = V^{*}\big (\tilde \Gamma (\Theta(\mu)(\hat x))\big ) V
= V^{*}\big ( \id \otimes \Phi_{\mu})(V(\hat x \otimes 1 )V^{*}) \big ) V = \hat x \otimes 1
\end{eqnarray*}
 for all  $\hat x \in L_{\infty}(\hat \G)$.  This shows  that  
 $\hat x \in  \h_{\Theta(\mu)}$
 and 
 \[
\hat x \otimes 1= \tilde \Gamma(\hat x)= 
\tilde \Gamma_{\Theta(\mu)}(\hat x ) \in  \tilde \Gamma_{\Theta(\mu)}
({{\mathcal{H}}_{\Theta(\mu)}})
\]
for all $\hat x\in L_{\infty}(\hat \G)$.
Moreover, since  ${\mathcal{H}}_\mu\subseteq {\mathcal{H}}_{\Theta(\mu)}$
and the restriction of $\tilde \Gamma_{\Theta(\mu)} $ to ${\h_{\mu}}$ is 
equal to $\Gamma_{\mu}$, we obtain
\[
\Gamma_{\mu}(\h_{\mu}) = \tilde \Gamma_{\Theta(\mu)}(\h_{\mu})\subseteq \tilde \Gamma_{\Theta(\mu)}
({{\mathcal{H}}_{\Theta(\mu)}}).
\]
Since $\tilde \Gamma_{\Theta(\mu)}$ is a unital $*$-homomorphism 
 from $\h_{\Theta(\mu)}$ into 
$\B(L_{2}(\G))\vtp \h_{\mu}$,  this implies that
\[
 \G  \ltimes _{\Gamma_{\mu}} \h_{\mu}  = (\Gamma_{\mu}({\mathcal{H}}_\mu)
 \cup ( L_{\infty}(\hat \G)\otimes 1))''
 = \tilde \Gamma_{\Theta(\mu)}(({\mathcal{H}}_\mu\cup L_{\infty}(\hat \G))'')
 \subseteq \tilde\Gamma_{\Theta(\mu)} ({\mathcal{H}}_{\Theta(\mu)}).
\]

Conversely, let $x \in {\mathcal{H}}_{\Theta(\mu)}$. 
Then  $\tilde \Gamma_{\Theta(\mu)} (x) \in \B(L_2(\G))\vtp \h_{\mu}$
and 
\begin{eqnarray*}
\beta(\tilde \Gamma_{\Theta(\mu)}(x)) &=& 
(\sigma V^{*} \sigma\otimes 1)\big((\chi\otimes \id)(\id \otimes \tilde\Gamma)
(\tilde\Gamma_{\Theta(\mu)} (x))\big)(\sigma V \sigma\otimes 1)\\
&=& (\sigma V^{*}\sigma\otimes 1)
\big((\sigma\otimes \id)(\tilde\Gamma\otimes \id)(\tilde\Gamma_{\Theta(\mu)}(x))
(\sigma\otimes  \id)\big)(\sigma V \sigma\otimes 1)\\
&=&(\sigma V^{*} \otimes 1)\big(( V \otimes 1)
(\tilde\Gamma_{\Theta(\mu)}(x)_{13})(V^{*}\otimes 1)\big)
(V \sigma\otimes 1)\\
&=& (\sigma\otimes 1)(\tilde\Gamma_{\Theta(\mu)} (x)_{13})(\sigma\otimes1)
= 1 \otimes \tilde\Gamma_{\Theta(\mu)} (x).
\end{eqnarray*}
Therefore, we have 
\[
\tilde\Gamma_{\Theta(\mu)} (x)\in (\B(L_2(\G))\vtp \h_{\mu})^\beta 
= \G\ltimes _{\Gamma_{\mu}}  \h_{\mu}  
\]
for all $x \in {\mathcal{H}}_{\Theta(\mu)}$.
This  completes the proof.
\end{proof} 

The following result is an immediate consequence of Theorem \ref {main}.

\begin{corollary}\label{inj}
 The crossed product von Neumann algebra $\G  \ltimes _{\Gamma_{\mu}} \h_{\mu}$ is injective.
\end{corollary}

%%%
%%%

%...............................

%................................

%%%%%%%%%%%%%%%%%%%%%%%%%%%%%%%%%%%%%%%%%
%%%%%%%%%%%%%%%%%%%%%%%%%%%%%%%%%%%%%%%%%

As pointed out by Izumi in \cite{II} as a concrete example for the crossed product formula he obtained for countable discrete groups,
in the case of random walks on the 
free group $\mathbb{F}_n$ with respect to the uniform distribution on the generators, the
result yields an identification of the Poisson boundary on the level of 
$\B(\ell_2(\mathbb{F}_n))$ ($n \geq 2$) with the Powers factor of type $III_{1/(2n-1)}$. 
Moreover, as noted in \cite{II}, the crossed product realization immediately implies 
amenability of the natural $G$-action on the Poisson boundary in the sense of Zimmer. 
The concept of an amenable action was first introduced by
Zimmer \cite{Zim} in the context of a measure class preserving action of a locally compact
second countable group on a standard Borel space. It was subsequently shown that
if $G$ is a second countable locally compact group and
$\alpha:G\curvearrowright X$ is a measure class preserving
action of $G$ on a standard probability space $(X,\nu)$,
then $\alpha$ is an amenable action if and only if the crossed product
$G  \ltimes _{\alpha} L_\infty(X,\nu)$ is injective.
This was generalized to the case
of actions of locally compact groups on von Neumann algebras
by Anantharaman-Delaroche \cite{dela}.
From this perspective, our result suggests that injectivity of the crossed product may provide
a notion of `Zimmer amenability' of a quantum group action.
We note that a concept of (topologically) amenable actions for discrete quantum groups $\G$ in the $C^*$-algebra
framework was defined by Vaes--Vergnioux in \cite {VV3}.
It is shown in \cite[Proposition 4.4]{VV3} that topological amenability of the action
on a unital nuclear $C^*$-algebra entails nuclearity of the (reduced) $C^*$-crossed product.

Our proof of Theorem \ref{main} gives a simple way to obtain the following.

\begin{proposition}\label{oz}
Let $G$ be a discrete group. Then the following are equivalent:
\begin{itemize}
\item[(i)]
$G$ is exact;
\item[(ii)]
for every probability measure $\mu$ on $G$, the reduced $C^*$-algebra
crossed product
$G  \ltimes _{\Gamma_{\mu}} \h_{\mu}$
is nuclear;
\item[(iii)]
for some probability measure $\mu$ on $G$, the crossed product
$G  \ltimes _{\Gamma_{\mu}} \h_{\mu}$ is nuclear.
\end{itemize}
\end{proposition}
\begin{proof}
(i) $\Rightarrow$ (ii): \ 
Denote $\mathcal{R}_\mu := \mathcal{E}_{\Theta(\mu)} (UC^*(G))$.
Since $\mathcal{E}_{\Theta(\mu)} (a) = \mathcal{E}_{\mu} (a)$ for all $a\in \ell_\infty(G)$,
and $\mathcal{E}_{\Theta(\mu)} (\hat a) = \hat a$ for all $\hat a\in C^*_r(G)$,
it follows from the proof of Theorem \ref{main} that
\begin{equation}\label{jit}
\mathcal{R}_\mu \,=\, \tilde \Gamma_{{\Theta(\mu)}}^{-1} (G  \ltimes _{\Gamma_{\mu}} \h_{\mu})\,.
\end{equation}
Moreover, we have $\mathcal{R}_\mu \subseteq UC^*(G)$,
and hence $\mathcal{E}_{\Theta(\mu)}$ restricts to a unital completely positive idempotent on $UC^*(G)$.
Since $G$ is exact, the uniform Roe algebra $UC^*(G)$ is nuclear \cite{Oz}.
Now it is easy to see from the completely positive factorization property
\begin{equation*}
\mathcal{R}_\mu \,\hookrightarrow\, UC^*(G) \rightarrow\, M_{n(\alpha)}\, \rightarrow \, UC^*(G)\, \rightarrow \, \mathcal{R}_\mu
\end{equation*}
that $\mathcal{R}_\mu$ (with its Choi--Effros product) is also nuclear.
But by (\ref{jit}), $\tilde \Gamma_{{\Theta(\mu)}}$ induces a completely isometric order isomorphism of the operator systems
$\mathcal{R}_\mu \cong G  \ltimes _{\Gamma_{\mu}} \h_{\mu}$.
Hence the latter is nuclear.\\
(ii) $\Rightarrow$ (iii): \ Obvious.\\
(iii) $\Rightarrow$ (i): \ This follows from the fact that $C^*_r(G)\subseteq G  \ltimes _{\Gamma_{\mu}} \h_{\mu}$.
\end{proof}

Ozawa proved in \cite{Oz} that a discrete group $G$ is exact if and only if its action
on the Stone--$\mathrm{\check{C}}$ech compactification $\beta G$ is amenable;
equivalently, $G$ admits an amenable action on some compact space.
However, if $G$ acts amenably on $\beta G$, of course it does not need to do so on an arbitrary compact space
(indeed, amenability of the trivial action on a one-point set is equivalent to amenability of the group itself).
Since $\beta G$ is often too large to be useful in applications, it is
interesting to find ``smaller" amenable $G$-spaces, when they exist.
This motivated the above result.

Finally, one may compare Proposition \ref{oz} with the following characterization of the amenability, due to Kaimanovich--Vershik \cite{1} (see also \cite{Ros} for non-discrete groups);
note that the implication (ii) $\Rightarrow$ (i) follows immediately from Proposition \ref{oz}.
\begin{theorem}
Let $G$ be a countable discrete group. Then the following are equivalent:
\begin{itemize}
\item[(i)]
$G$ is amenable;
\item[(ii)]
$\h_\mu = \C1$ for some probability measure $\mu$ on $G$.
\end{itemize}
\end{theorem}
\noindent
We remark that requiring condition (ii) to hold for all adapted probability measures $\mu$ on $G$ (i.e., such that the subgroup generated by the support of $\mu$ is dense in $G$)
is a much stronger property than amenability, called the Liouville property; there exist even solvable groups without the Liouville property \cite[Proposition 6.1]{1}.

\bibliographystyle{plain}

\begin{thebibliography}{99}










\bibitem {dela} C. Anantharaman-Delaroche,
\textit{Action moyennable d'un groupe localement compact sur une alg\`{e}bre de von Neumann},
Math. Scand.  \textbf{45}  (1979), no. 2, 289--304.





\bibitem {Chu-Lau} C.-H. Chu \and A. T.-M. Lau,
\textit{Harmonic functions on groups and Fourier algebras}, Lecture Notes in Mathematics,
\textbf{1782}, Springer-Verlag, Berlin, 2002.




\bibitem {EKR} E. G. Effros, J. Kraus \and Z.-J. Ruan,
\textit{On two quantized tensor norms}, Operator Algebras,
Mathematical Physics, and Low Dimensional Topology (Istanbul
1991),  Res. Notes Math. 5, A K Peters,
Wellesley, MA 1993, 125--145.



\bibitem {ERbook} E. G. Effros \and Z.-J. Ruan, \textit{Operator spaces},
London Math. Soc. Monographs, New Series {\bf 23}, Oxford
University Press, New York, 2000.

\bibitem {Enock} 
M. Enock, {\it Measured quantum groupoids in action}, M\'em. Soc. Math. Fr. (N.S.), no. 114 (2008).


\bibitem{Furst} H. Furstenberg,
\textit{Boundary theory and stochastic processes on homogeneous spaces},
Harmonic analysis on homogeneous spaces (Proc. Sympos. Pure Math., Vol. XXVI, Williams Coll., Williamstown, Mass., 1972),
 Amer. Math. Soc., Providence, R.I., 1973, 193--229.



\bibitem{Gha} F. Ghahramani,
\textit{Isometric representation of $M(G)$ on $B(H)$}, Glasgow Math. J.
 \textbf{23} (1982), 119--122.


\bibitem{HK} U. Haagerup \and J. Kraus,
\textit{Approximation properties for group $C^*$-algebras and group von Neumann algebras}, 
Trans. Amer. Math. Soc. {\bf 344} (1994), 667--699.





\bibitem {I} 
M. Izumi, \textit{Non-commutative Poisson boundaries and compact quantum group actions}, 
Adv. Math. \textbf{169} (2002), no. 1, 1--57. 

\bibitem {II} 
M. Izumi, \textit{Non-commutative Poisson boundaries}, in: \textit{Discrete geometric analysis},
Contemp. Math., 347, Amer. Math. Soc., Providence, RI, 2004, 69--81. 

\bibitem {III} 
M. Izumi, \textit{$E_0$-semigroups: around and beyond Arveson's work}, 
J. Operator Theory \textbf{68} (2012), no. 2, 335-–363. 

\bibitem {INT}
M. Izumi, S. Neshveyev \and L. Tuset,
 \textit{Poisson boundary of the dual of ${SU}_q(n)$}, Comm. Math. Phys.
\textbf{262} (2006), no. 2,  505--531.


\bibitem {JN} 
W. Jaworski \and M. Neufang, \textit{The Choquet--Deny equation in a Banach space},
Canad. J. Math. \textbf{59} (2007), no. 4, 795--827.




\bibitem {JNR} M. Junge, M. Neufang \and Z.-J. Ruan,
\textit{A representation theorem for locally compact quantum
groups}, Int. J. Math. \textbf{20} (2009), 377--400.



\bibitem {1} V. A. Kaimanovich \and A. M. Vershik,
{\it Random walks on discrete groups: boundary and entropy}, Ann.  Probability,
\textbf{11} (1983), 457--490.



\bibitem{Kus} J. Kustermans, \textit{Locally compact quantum groups in 
the universal setting}, International J. Math. {\textbf 12} (2001), 289--338.


\bibitem {KV1} J. Kustermans \and S. Vaes, \textit{Locally compact quantum groups},
Ann. Sci. Ecole Norm. Sup. \textbf{33} (2000), 837--934.

\bibitem {KV2} J. Kustermans \and S. Vaes, \textit{Locally compact quantum groups
in the von Neumann algebraic setting}, Math. Scand. \textbf{92}
(2003), 68--92.
 



\bibitem {NRS} M. Neufang, Z.-J. Ruan \and N. Spronk,
\textit{Completely isometric representations of $M_{cb}A(G)$ and
$UCB(\hat G)^*$}, Trans. Amer. Math. Soc. \textbf{360} (2008),
1133--1161.

\bibitem {NR} M. Neufang \and V. Runde, \textit{Harmonic operators: the dual perspective},
Math. Z. \textbf{255} (2007), 669--690.


\bibitem {Oz} N. Ozawa,
\textit{Amenable actions and exactness for discrete groups}, C. R. Acad. Sci. Paris Sér. I Math. 
{\bf 330} (2000), no. 8, 691--695.





\bibitem {Ros} J. Rosenblatt,
 \textit{Ergodic and mixing random walks on locally compact groups}, Math. Ann.
\textbf{257} (1981), no. 1,  31--42.


\bibitem{Ruan1992}  Z.-J. Ruan, {\it On the predual of dual algebras}, J. Operator Theory
 {\bf 27} (1992), 179--192.

\bibitem{Sto} E. St{\o}rmer, {\it Regular abelian Banach algebras of linear maps of operator algebras}, J. Funct.
Anal. {\bf 37} (1980), 331-–373. 

\bibitem {Tom} R. Tomatsu,
\textit{A characterization of right coideals of quotient type and its application to classifcation of Poisson
boundaries},
Comm. Math. Phys. {\bf 275} (2007), no. 1, 271--296.





\bibitem {VV2} S. Vaes \and N. Vander Vennet,
\textit{Poisson boundary of the discrete quantum group $\widehat{A_u(F)}$}, Compos. Math.
 \textbf{146} (2010), no. 4,  1073--1095.



\bibitem {VV3} S. Vaes \and R. Vergnioux,
\textit{The boundary of universal discrete quantum groups, exactness, and factoriality}, Duke Math. J.
 \textbf{140} (2007), no. 1,  35--84.





\bibitem {Zim} R. Zimmer, \textit {Amenable ergodic group actions
and an application to Poisson boundaries of Random walks}, 
J. Funct. Anal. {\bf 27} (1978), 350--372.






\end{thebibliography}

\end{document}